\documentclass{amsart}

\usepackage{amssymb}

\newtheorem{theorem}{Theorem}
\newtheorem{lemma}[theorem]{Lemma}
\newtheorem{proposition}[theorem]{Proposition}
\newtheorem{examp}{Example}

\newtheorem{corollary}[theorem]{Corollary}
\newtheorem{remar}[theorem]{Remark}
\newtheorem{ex}{Exercise}
\newtheorem{prob}{Open Problem}
\newcommand{\diams}{\unskip\nobreak\hfil\penalty50%
\hskip1em\hbox{}\nobreak\hfil%
$\diamondsuit$\parfillskip=0pt\finalhyphendemerits=0}

\newcommand{\bfind}[1]{\index{#1}{\bf #1}}

\newcommand{\n}{\par\noindent}

\newcommand{\sn}{\par\smallskip\noindent}
\newcommand{\mn}{\par\medskip\noindent}
\newcommand{\bn}{\par\bigskip\noindent}
\newcommand{\pars}{\par\smallskip}
\newcommand{\parm}{\par\medskip}

\newcommand{\sep}{^{\rm sep}}
\newcommand{\chara}{\mbox{\rm char}\,}
\newcommand{\trdeg}{\mbox{\rm trdeg}\,}
\newcommand{\appr}{\mbox{\rm appr}\,}
\newcommand{\dist}{\mbox{\rm dist}\,}
\newcommand{\ndd}{\mbox{\rm ndd}\,}

\newcommand{\cal}{\mathcal}
\newcommand{\subsetuneq}{\mathrel{\raisebox{.8ex}{\footnotesize%
$\displaystyle\mathop{\subset}_{\not=}$}}}

\newcommand{\modd}{\mbox{\rm mod }\,}

\newcommand{\N}{\mathbb{N}}

\begin{document}
\title[Number of distinct distances]{Counting the number of distinct distances of elements in valued field extensions}
%
\author{Anna Blaszczok and Franz-Viktor Kuhlmann}
\address{Institute of Mathematics, University of Silesia in Katowice, Bankowa 14,
40-007 Katowice, Poland}
\email{anna.blaszczok@us.edu.pl, fvk@math.us.edu.pl}
%
\date{24.\ 4.\ 2017}

%
\begin{abstract}
The defect of valued field extensions is a major obstacle in open problems in resolution of singularities and in
the model theory of valued fields, whenever positive characteristic is involved. We continue the detailed study of
defect extensions through the tool of distances, which measure how well an element in an immediate extension can be
approximated by elements from the base field. We show that in several situations the number of essentially distinct
distances in fixed extensions, or even just over a fixed base field, is finite, and we compute upper bounds. We
apply this to the special case of valued functions fields over perfect base fields. This provides important
information used in forthcoming research on relative resolution problems.
\end{abstract}
\maketitle
%
%
%
\section{Introduction}
By $(L|K,v)$ we denote a field extension $L|K$ where $v$ is a valuation
on $L$ and $K$ is endowed with the restriction of $v$. The valuation
ring of $v$ on $L$ will be denoted by ${\cal O}_L\,$, and that on $K$ by
${\cal O}_K\,$. The value group of $(L,v)$ will be denoted by $vL$, and
its residue field by $Lv$. The value of an element $a$ will be denoted
by $va$, and its residue by $av$.

The \bfind{defect}, also known as \bfind{ramification deficiency}, of finite extensions $(L|K,v)$ of valued fields
is a phenomenon that only appears when the residue field $Kv$ has positive characteristic. It is a main obstacle to
the solution of deep open problems in positive characteristic, such as:
\sn
$\bullet$ \ local uniformization (the local form of resolution of singularities), which is not known for arbitrary
dimension in positive characteristic,
\sn
$\bullet$ \  the model theory of valued fields, in particular the open question whether Laurent series fields over
finite fields have a decidable theory.
\sn
Both problems are linked through the structure theory of valued function fields, in which it is essential to tame
the defect, as well as wild ramification, cf.\ \cite{[Ku2],[Ku2a],[Ku5],[Ku6]}. While implicitly known through the
work of algebraic geometers and model theorists since the 1950s, the connection of the defect with the
problem of local uniformization and the model theory of valued fields with positive residue characteristic has been
pointed out in detail in the cited works of the second author. Defects also appear in crucial examples, as in the
paper \cite{[C--P]}.

Using tools of ramification theory, the study of extensions of valued fields of residue characteristic $p>0$
with nontrivial defect can be reduced to the study of normal extensions of degree $p$ with nontrivial defect. Such
extensions are immediate. An arbitrary extension $(L|K,v)$ of valued fields is \bfind{immediate} if the canonical
embeddings of $vK$ in $vL$ and of $Kv$ in $Lv$ are onto. As a consequence, for every $a\in L\setminus K$ the set
\[
v(a-K) \>:=\> \{v(a-c)\mid c\in K\}
\]
does not have a maximal element; this follows from \cite[Theorem 1]{[Ka]}. If $a$ is an element of any valued field
extension of $(K,v)$ such that $v(a-K)$ has no maximal element, then this set is an initial segment of $vK$. We
associate with it a cut
in the divisible hull $\widetilde{vK}$ of $vK$ by taking as the lower cut set the smallest initial segment in
$\widetilde{vK}$ which contains $v(a-K)$. This cut is called the \bfind{distance of $a$ over $K$} and denoted by
$\dist(a,K)$. For more details, see Section~\ref{sectdist}.

Distances can be used to classify defect extensions. If an extension $L|K$ of degree $p$ is Galois
and the field $K$ is itself of characteristic $p$, then $L|K$ is an \bfind{Artin--Schreier extension}, that is, $L$
is generated over $K$ by an element $\vartheta$ such that
\begin{equation}                            \label{ASg}
\vartheta^p-\vartheta\;\in\; K\>;
\end{equation}
we call $\vartheta$ an \bfind{Artin-Schreier generator} of the
extension. If such an extension of a valued field $(K,v)$ has nontrivial
defect, then the extension of the valuation $v$ from $K$ to $L$ is unique and $(L|K,v)$ is immediate (see
Lemma~\ref{Gepd} below); we call it an \bfind{Artin--Schreier defect extension}. A classification of
Artin--Schreier defect extensions is introduced in \cite{[Ku3]}, and it is shown that the classification can be read
off from the distance $\dist(\vartheta,K)$ of the Artin-Schreier generator. In a collaboration of the second author
with O.~Piltant (\cite{[KuP]}) the question arose how many distinct distances of generators of Artin-Schreier defect
extensions exist over a fixed $(K,v)$ (in particular, whether this number is finite at all).

If $c\in K$, then $v(ca-K)=\{vc+v(a-c)\mid c\in K\}=:vc+v(a-K)$, which means that the cut $\dist(a,K)$ is just
shifted by adding $vc$ to all elements of the lower cut set; we then write
\begin{equation}                              \label{shiftcut}
\dist(ca,K)\>=\>vc+\dist(a,K)\>.
\end{equation}
We do not regard $\dist(a,K)$ and $\dist(ca,K)$ as essentially distinct, so we will actually ask for the number of
distances that are distinct modulo $vK$. In Section~\ref{sectdistAS} we give an answer under certain finiteness
assumptions, see Theorem~\ref{r+m}. These conditions hold for instance in the following situation:

\begin{theorem}                       \label{MT1}
Take a valued function field $(K|K_0,v)$ over a perfect trivially valued
base field $(K_0,v)$. Then the number of distinct distances of elements in
Artin-Schreier defect extensions modulo $vK$ is bounded by $2\cdot\trdeg K|K_0\,$.
\end{theorem}
This answers a question from Olivier Piltant; results of this type are a crucial tool in \cite{[KuP]}.

\pars
More generally, we would like to count all the essentially distinct distances over $K$ of all elements
$a\in\tilde{K}$ for which $v(a-K)$ has no maximal element. But it seems unlikely that we will get a finite number
if we allow the elements $a$ to attain arbitrarily large degree over $K$, so we need again some conditions.
The first way to impose suitable conditions is to restrict the scope to all elements $a\in L$ where $L|K$ is a
finite extension such that the extension of $v$ from $K$ to $L$ is unique. For this case, we obtain in
Section~\ref{sectdistLK} an upper bound in terms of the defect of the
extension $(L|K,v)$ and its ramification index $(vL:vK)$, see Theorem~\ref{nonhens}.

Another approach is to limit the scope to all $a\in\tilde{K}$ of bounded degree over $K$. It is an open problem
whether the number of essentially distinct distances in this case is always finite and
to compute an upper bound for it, even under the finiteness conditions of
Theorem~\ref{r+m}. However, we are able to show that under
these finiteness conditions, the number of distances that are distinct modulo
$\widetilde{vK}$ is always finite; we give an upper bound in Theorem~\ref{ndd_i}.

Note that there are examples of valued fields of rank 1, but infinite $p$-degree, where even the number of
distances of elements in immediate purely inseparable extensions of degree $p$ (and of elements in Artin-Schreier
defect extensions) that are distinct modulo $\widetilde{vK}$ is infinite.

\bn
%
%
\section{Preliminaries}
For general facts from valuation theory, we refer the reader to \cite{[En],[E--P],[Ri],[ZS2]}.
%

%
%
\subsection{Defect}   \label{sectied}
Take a finite normal extension $L|K$ and a valuation $v$ on $K$. Then
$v$ has finitely many distinct extensions $v_1,\ldots,v_g$ to $L$. All
of them have the same ramification index $(v_iL:vK)$, which we will
denote by $e$, and all of them have the same inertia degree $[Lv_i:Kv]$,
which we will denote by $f$. Then we have the fundamental equality
\begin{equation}                   \label{feq}
[L:K]\>=\> d\cdot e\cdot f\cdot g\>,
\end{equation}
where by the Lemma of Ostrowski (cf.\ \cite[Th\'eor\`eme 2, p.~236]{[Ri]}) or \cite[Corollary to Theorem~25,
Section G, p.~78]{[ZS2]}), $d$ is a power $\geq 1$ of the residue
characteristic $\chara Kv$ if this is positive, and equal to 1 otherwise.
If $d>1$, then we speak of \bfind{nontrivial defect}.
If in addition $L|K$ is an extension of prime degree, then it follows from (\ref{feq})
that $[L:K]=d=\chara Kv>0$ and $e=f=g=1$, that is, there is a unique
extension of $v$ from $K$ to $L$ and $(L|K,v)$ is immediate. We have proved:

\begin{lemma}                        \label{Gepd}
If $L$ is a normal extension of prime degree $p$ of $(K,v)$ with nontrivial defect, then the extension of $v$
from $K$ to $L$ is unique and $(L|K,v)$ is immediate.
\end{lemma}

\parm
We will almost always consider extensions $(L|K,v)$ for which the extension of $v$ from $K$ to $L$ is unique.
We will call such extensions \bfind{uv--extensions} in short; they
are necessarily algebraic extensions. Note that every purely inseparable algebraic
extension is a uv--extension.

For a finite uv--extension $(L|K,v)$, we can define its \bfind{defect}
even if the extension is not normal:
\[
d(L|K,v)\>:=\>\frac{[L:K]}{(vL:vK)[Lv:Kv]}\>.
\]
By the Lemma of Ostrowski, this is a power of $p$ (including
$p^0=1$), where $p=\chara Kv$ if this is positive, and $p=1$ otherwise (this is called the \bfind{characteristic exponent} of $Kv$).
The extension is called \bfind{defectless} if $d(L|K,v)=1$; otherwise, we call it a \bfind{defect extension}.
Note that if $(L|K,v)$ is a defect extension of prime degree $p$, then $p=\chara Kv$. We note:
\begin{lemma}                     \label{imuv}
If $(L|K,v)$ is a finite immediate uv--extension, then $[L:K]$ is a power of $p$ and $d(L|K,v)=[L:K]$.
\end{lemma}

A valued field $(K,v)$ is \bfind{henselian} if it satisfies Hensel's Lemma, or equivalently, if the extension of
$v$ to the algebraic closure $\tilde{K}$ of $K$ is unique (i.e., $\tilde{K}$ is a uv--extension of $(K,v)$).
In this case, $v$ extends uniquely to each algebraic extension of $K$.
Every algebraically closed valued field is trivially henselian.

Every valued field $(K,v)$ admits a \textbf{henselization}, that is, a
minimal henselian extension of $(K,v)$, in the sense that it admits a unique valuation preserving embedding over
$K$ in every other henselian extension of $(K,v)$. In particular, if $w$ is any extension of $v$
to $\tilde{K}$, then $(K,v)$ has a unique henselization in $(\tilde{K},w)$, as it is the decomposition field of the
normal extension $(K\sep|K,v)$, where $K\sep\subseteq\tilde{K}$ is the separable-algebraic closure of $K$.

Henselizations of $(K,v)$ are unique up to valuation preserving isomorphism over $K$. Moreover, they are always
immediate separable-algebraic extensions of $(K,v)$ (cf.\ \cite[Theorem~17.19]{[En]}).
A valued field is henselian if and only if it is equal to any (and thus all) of its henselizations.

\pars
The following fact is Lemma~2.1 of \cite{[B--K]}:
\begin{lemma}                               \label{uvdisj}
An algebraic extension $(L|K,v)$ is a uv--extension if and only if
for an arbitrary henselization $K^h$ of $(K,v)$, the extensions $L|K$
and $K^h|K$ are linearly disjoint.
\end{lemma}

\mn
{\bf For the remainder of this paper, we fix an extension of $v$ from $K$ to $\tilde{K}$.} This will also fix the
henselization of $(K,v)$. Therefore, we will speak of {\it the} henselization of $(K,v)$, and denote it by $(K^h,v)$.

Since the henselization is an immediate extension and the compositum $L.K^h$ of $L$ and $K^h$ lies in $L^h$ (in fact, it is equal to $L^h$), this lemma yields:

\begin{lemma}                               \label{hensdef}
For every finite uv--extension $(L|K,v)$,
\[
d(L|K,v)\>=\>d(L.K^h|K^h,v)\>.
\]
\end{lemma}

\mn
%
%
\subsection{Distances}                                   \label{sectdist}
Take an arbitrary extension $(L|K,v)$ of valued fields and $a\in L\setminus K$.
There are several possible definitions for the \bfind{distance} of $a$ from $K$ that have been used in papers
by the first author. We choose the definition that is most suitable for our purposes in this paper.
%

By $\dist(a,K)$ we denote the cut induced by the set $v(a-K)\cap\widetilde{vK}$
in the divisible hull $\widetilde{vK}$ of $vK$. Namely, the lower cut set of $\dist(a,K)$ is the
smallest initial segment that contains $v(a-K)\cap\widetilde{vK}$. This definition
is slightly different from the one introduced in \cite{[Ku3]} and \cite{[KuV]}. There, we have used the cut in
$\widetilde{vK}$ induced by the subset $v(a-K)\cap vK$ to define $\dist(a,K)$.
%
%
A detailed study of the new notion of distance and a comparison with the former notion can be found in \cite{[B]}.
Note that when $v(a-K)\subseteq vK$, the two notions coincide.

\pars
Our definition enables us to compare $\dist(a,K)$ with $\dist(a,L)$
when $(L|K,v)$ is an algebraic extension since then, both $\dist(a,K)$
and $\dist(a,L)$ are cuts in the same ordered abelian group $\widetilde{vK}=\widetilde{vL}$.
Then $\dist(a,K)<\dist(a,L)$ will mean that the left cut set of $\dist(a,K)$
is a proper subset of that of $\dist(a,L)$.

\pars
The following is Lemma~3.9 of \cite{[B]}.

\begin{lemma}                               \label{distdist}
Take algebraic extensions $(L|K,v)$ and $(L(a)|L,v)$. Then $\dist(a,K)\leq\dist(a,L)$.
If $\dist(a,K)<\dist(a,L)$, then there is $b\in L$ such that
\[
v(a-b)\> >\>v(a-K)\>=\>v(b-K) \quad\mbox{ and }\quad
\dist(b,K)\>=\>\dist(a,K)\;.
\]
\end{lemma}

\pars
If $(L|K,v)$ is an arbitrary valued field extension and $a\in L$, then
we will say that $a$ is \bfind{weakly immediate over $K$} if $v(a-K)$ has no maximal
element. In the language of pseudo Cauchy sequences, this means that $a$
is a pseudo limit of a pseudo Cauchy sequence (also called ``pseudo convergent sequence'' in \cite{[Ka]})
in $(K,v)$ that has no pseudo limit in $K$. In the language used in \cite{[KuV]} it means that the
approximation type of $a$ over $K$ is immediate. Note that this does \emph{not} imply that
the extension $(K(a)|K,v)$ is immediate (cf.\ \cite[Example~3.17]{[B]}). But conversely, by what
we have already said in the introduction, every
element in an immediate extension of $(K,v)$ is weakly immediate over $K$.
Observe that if $a$ is weakly immediate over $K$ then $\infty\notin v(a-K)$, that is, $a\notin K$.

\begin{lemma}                               \label{2.7}
Take a finite defectless uv--extension $(L|K,v)$. Then the following assertions hold.
\sn
a) \ For every $b\in L\setminus K$, the set $v(b-K)$ has a maximal element.
\sn
b) \ Every $a\in\tilde{L}=\tilde{K}$ that is weakly immediate over $K$ is also weakly immediate over $L$, and
\[
\dist(a,L)\>=\>\dist(a,K)\;.
\]
\end{lemma}
\begin{proof}
a): \ This follows from Proposition 3.12 and Lemma 3.10 of \cite{[B]}.

\sn
b): \ This is Corollary~3.11 of \cite{[B]}.
%

\end{proof}

To obtain another important distance equality, we need the following theorem from \cite{[Ku4]}:

\begin{theorem}                             \label{apprMT2}
Take $K^h$ to be the henselization  of $K$ in $(\tilde{K},v)$. Take $a\in
\tilde{K}\setminus K$ and assume that for some $b\in K^h$,
\[
v(a-b)\;>\;v(a-K)\>.
\]
Then $K^h$ and $K(a)$ are not linearly disjoint over $K$.
\end{theorem}

\begin{lemma}                               \label{dKdKh}
Take an algebraic uv--extension $(L|K,v)$. Then for all $a\in L\setminus K$ which are weakly immediate over $K$,
\begin{equation}                            \label{vx-Khens}
v(a-K^h)\>=\>v(a-K) \quad\mbox{ and }\quad \dist(a,K^h)\>=\>\dist(a,K)\;.
\end{equation}
\end{lemma}
\begin{proof}
Take $a\in L\setminus K$ and suppose that $v(a-K)\subsetuneq v(a-K^h)$. Then there is an element $b\in K^h$ such
that $v(a-b)>v(a-K)$. But then by Theorem~\ref{apprMT2}, $K(a)|K$ and hence also $L|K$ is not linearly disjoint from
$K^h|K$, a contradiction to Lemma~\ref{uvdisj}. So we have that $v(a-K)=v(a-K^h)$, which implies the equality of the
distances.
\end{proof}

%
%
\subsection{Weakly and strongly immediate elements}
We have already defined what it means for an element in an extension of $(K,v)$ to be weakly immediate
over $(K,v)$. A useful stronger property is the following. Take any extension $(L|K,v)$ of valued fields and an
element $a\in L\setminus K$. Then we will say that $a$ is \bfind{strongly immediate over $K$} if $v(a-K)$ has no
maximal element and in addition, for every polynomial $g\in K[X]$ of degree $<[K(a):K]$ there is
$\alpha\in v(a-K)$ such that for all $c\in K$ with $v(a-c)\geq\alpha$, the value $vg(c)$ is fixed.
\begin{lemma}                                   \label{si}
If the element $a$ is strongly immediate over $K$, then $(K(a)|K,v)$ is immediate. If in addition, $(K(a)|K,v)$ is
a uv--extension, then $[K(a):K]=d((K(a)|K,v)=p^k$ for some $k\geq 1$, with $p$ the characteristic exponent of $Kv$.
\end{lemma}
\begin{proof}
For the first assertion, see \cite[Lemma~5.3]{[KuV]}.
The second assertion follows from the first together with Lemma~\ref{imuv}.
\end{proof}

In general, even if $(K(a)|K,v)$ is a uv--extension and $a$ is weakly immediate over $K$, the extension may not
be immediate and $a$ may not be strongly immediate over $K$. But this holds if the degree $[K(a):K]$ is a prime:
\begin{lemma}                               \label{EF}
Take a uv--extension $(K(a)|K,v)$ of prime degree $p$ with its generator $a$ weakly immediate
over $K$. Then $(K(a)|K,v)$ is immediate and $a$ is  strongly immediate over $K$.
\end{lemma}
\begin{proof}
By~\cite[Lemma~9]{[Ku3]}, $(K(a)|K,v)$ is immediate. Note that by Lemma~\ref{imuv}, $p = \chara Kv > 0$.

Suppose that there is a polynomial $g\in K[X]$ of degree $<p$ for which
there is no $\alpha\in v(a-K)$ such that the value $vg(c)$ is fixed for
all $c\in K$ with $v(a-c) \geq\alpha$. Since $v(a-K)=v(a-K^h)$ by (\ref{vx-Khens}), there is no $\alpha\in v(a-K^h)$
such that the value $vg(c)$ is fixed for all $c\in K^h$ with $v(a-c) \geq\alpha$. Take $f\in K^h[X]$ to be of
minimal degree with this property. As $\deg f\leq\deg g<p$, it follows from \cite[Proposition~6.5]{[KuV]} that
$\deg f=1$. Hence $f(X)=X-b$ for some $b\in K^h$.

Since $v(a-K^h)=v(a-K)$ has no maximal element, we can choose some $\alpha\in v(a-K^h)$ with $\alpha> v(a-b)$.
Take any $c\in K^h$ such that
$v(a-c)\geq\alpha$. Then $v(c-b)=\min\{v(c-a),v(a-b)\}=v(a-b)$, so the value $vf(c)$ is fixed for all such $c$.
This contradicts our choice of $f$ and shows that a polynomial $g$ as chosen in the beginning cannot exist.
\end{proof}

\begin{lemma}                                 \label{wisi}
Take a henselian field $(K,v)$ and an element $a\in \tilde{K}$ which is weakly immediate over $K$. If $a$ is not
strongly immediate over $K$, then
there is an immediate extension $(K(b)|K,v)$ with $\dist(b,K)=\dist(a,K)$ and $[K(b):K]<[K(a):K]$.
\end{lemma}
\begin{proof}
Using the notions of \cite{[KuV]}, we argue as follows. Since $v(a-K)$ has no maximal element, the approximation
type $\appr(a,K)$ is immediate by \cite[Lemma~4.1 a)]{[KuV]}. Take $g$ to be an associated minimal polynomial for
$\appr(a,K)$. Since the extension $(K(a)|K,v)$ is not strongly immediate, we have that $\deg g<[K(a):K]$. Take
$b\in \tilde{K}$ to be a root of $g$. Then \cite[Theorem~6.4]{[KuV]} shows that there is an extension $w$ of $v$
from $K$ to $K(b)$ such that $(K(b)|K,w)$ is immediate and $\appr(b,K)=\appr(a,K)$. Since $(K,v)$ is henselian, $w$
and $v$ must agree on $K(b)$, showing that $(K(b)|K,v)$ is immediate. The equality of the approximation types
implies that $v(b-K)=v(a-K)$, which in turn implies that $\dist(b,K)=\dist(a,K)$.
\end{proof}

%
%
\subsection{The ramification field}
For general ramification theory, see \cite{[E--P]} or \cite{[Ku8]}.
For information on tame valued fields, see \cite{[Ku7]}. We
will summarise here the main properties of the ramification field that we will use.

Let $(N|K,v)$ be a normal algebraic extension of henselian fields. We take
the \bfind{ramification field} $V$ of this extension to be the fixed field of the ramification group
$\{\sigma\in \mbox{\rm Aut}(N|K)\mid 0\ne x\in {\mathcal O}_L\,\Rightarrow\, v(\sigma x - x)>vx\}$
of the automorphism group of $N|K$ {\it in the maximal separable subextension of $N|K$}.

The \bfind{absolute ramification field} of a henselian field $(K,v)$ is the ramification field of the
normal algebraic extension $(K\sep|K,v)$, where $K\sep$ denotes the separable-algebraic closure of $K$.

\begin{lemma}                               \label{overr}
Take a normal extension $(N|K,v)$ of henselian fields with residue
characteristic $p>0$. Then its ramification field $V$ has the following properties:
\sn
a) The extension $V|K$ is separable.
\n
b) Every subextension of $N|V$ is a tower of normal extensions of degree $p$.
\n
c) The valued field extension $(V|K,v)$ is tame and hence every  finite subextension $(E|K,v)$ of $(V|K,v)$ is defectless.
\n
d) For every finite subextension $L|K$ of $N|K$,
\[
d(L|K,v)\>=\>d(L.V|V,v)\>.
\]
e) For all $a\in N \setminus K$ weakly immediate over $K$,
\[
\dist(a,V)\>=\>\dist(a,K)\;.
\]
\end{lemma}
\begin{proof}
Assertion a) follows from our definition.

Assertion b) follows from the
fact that the ramification group is a $p$-group (cf.\ \cite[Theorem 5.3.3]{[E--P]} and the proof of
\cite[Lemma~2.9]{[Ku3]}).

For assertion c), note that $V$ is a subfield of the absolute ramification field $K^r$
of $(K,v)$, which by part b) of \cite[Lemma 2.13]{[Ku7]} is a tame extension of $(K,v)$. Hence by part a) of the
same lemma, also $V$ is a tame extension of $(K,v)$. Thus every finite subextension $(E|K,v)$ of the tame extension
$(V|K,v)$ is defectless. In view of this, the equality of the defects follows from \cite[Proposition~2.8]{[Ku3]}.

For the proof of d) suppose that $\dist(a,V)>\dist(a,K)$. Then by Lemma~\ref{distdist} there is an element $b\in V$
such that $\dist(a,K)=\dist(b,K)$. On the other hand, $(K(b)|K,v)$ is a defectless uv--extension, by part c).
Together with part a) of Lemma~\ref{2.7} this contradicts the fact that $a$ is weakly immediate over $K$.
\end{proof}

\bn
%
%
%
\section{The number of distinct distances in a given valued field extension}   \label{sectdistLK}
Take a finite (not necessarily immediate) uv--extension $(L|K,v)$.
We wish to count the number of distances appearing in this extension
that are distinct modulo $vK$.
We define
\[
\ndd(L|K,v)
\]
to be the minimal $m\geq 0$ such
that there are elements $a_1,\ldots,a_m\in L\setminus K$ so that each
$a_i$ is weakly immediate over $K$ and for every $b\in L\setminus K$ for which
$v(b-K)$ has no maximal element, there is $i\in \{1,\ldots,m\}$ and
$\alpha\in vK$ with
\[
\dist(b,K)\>=\>\alpha+\dist(a_i,K)\>,
\]
that is, $\dist(b,K)$ and $\dist(a_i,K)$ are equal modulo $vK$.
If there is no such $b$ (which in particular is the case when $(L|K,v)$ is defectless, according to part a) of
Lemma~\ref{2.7}), then we set $\ndd(L|K,v)=0$. We will see that such a number $m$ always exists.

Similarly,
\[
\ndd^*(L|K,v)
\]
shall denote the number of distances appearing in $(L|K,v)$ that are distinct modulo $\widetilde{vK}$. Observe that
\begin{equation}
\ndd^*(L|K,v)\>\leq\>\ndd(L|K,v)\;\mbox{ and }\;\ndd^*(L|K,v)=0\>\Leftrightarrow\>\ndd(L|K,v)=0 \>.
\end{equation}

We note:
\begin{lemma}                  \label{ndd0}
Take any algebraic extension $(L|K,v)$ of valued fields with subextension $(L_0|K,v)$. Then $\ndd(L_0|K,v)=0$ if
and only if $\dist(a,K)=\dist(a,L_0)$ for every $a\in L$ which is weakly immediate over $K$.
\end{lemma}
\begin{proof}
Assume first that $\ndd(L_0|K,v)=0$ and take $a\in L\setminus K$ weakly immediate over $K$. If $\dist(a,K)\ne
\dist(a,L_0)$, then $\dist(a,L_0)>\dist(a,K)$ and by Lemma~\ref{distdist} there is $b\in L_0$ such that
$\dist(a,K)=\dist(b,K)$. But then $v(a-K)$ has no maximal element, contradicting our assumption that
$\ndd(L_0|K,v)=0$.

Now assume that $\ndd(L_0|K,v)>0$ and take $a\in L_0\setminus K$ weakly immediate over $K$. Since $a\in L_0\,$,
it follows that $\dist(a,L_0)>\dist(a,K)$.
\end{proof}

\begin{lemma}                           \label{nddKK'}
Take a finite uv--extension $(L|K,v)$ and an algebraic extension $(K'|K,v)$ such that
$(vK':vK)<\infty$ and $\dist(a,K)=\dist(a,K')$ for all $a\in L$. Then
\begin{eqnarray*}
\ndd(L|K,v)&\leq&\ndd(L.K'|K',v)\cdot (vK':vK)\>,\\
\ndd^*(L|K,v)&\leq&\ndd^*(L.K'|K',v)\>.
\end{eqnarray*}
\end{lemma}
\begin{proof}
Set $n=(vK':vK)$ and choose representatives $\beta_1,\ldots,\beta_n\in vK'$ of the distinct cosets in $vK'/vK$.
If two distances $\dist(a_1,K)$ and $\dist(a_2,K)$ are equal modulo $vK'$ then there is $i\in\{1,\ldots,n\}$ and
$\alpha\in vK$ such that $\dist(a_1,K)=\alpha+\beta_i+\dist(a_2,K)$, where the latter is equal to
$\beta_i+\dist(a_2,K)$ modulo $vK$. This shows that the maximum number of distances that are distinct modulo $vK$
but equal modulo $vK'$ is $n$, which proves the first inequality.

The second inequality follows from the fact that all $\beta_i$ lie in $\widetilde{vK}$.
\end{proof}

The next lemma computes $\ndd(K(a)|K,v)$ for uv--extensions
$(K(a)|K,v)$ with $a$ strongly immediate over $K$. We derive it from \cite[Lemma~8]{[Ka]} and
\cite[Lemma~5.2]{[KuV]}. We use the Taylor expansion
\begin{equation}                             \label{Taylorexp}
f(X) = \sum_{i=0}^{n} f_i(c) (X-c)^i
\end{equation}
where $f_i$ denotes the $i$-th formal derivative of $f$.

\begin{lemma}                                \label{Kap}
Take a finite uv--extension $(K(a)|K,v)$ such that $a$ is strongly immediate over $K$.
Following Lemma~\ref{si}, we write $[K(a):K]=p^k$ for some $k\geq 1$. Then for every nonconstant polynomial
$f\in K[X]$ of degree $<p^k$ there are $\gamma\in v(a-K)$ and ${\bf h}=p^\ell$ with
$0\leq\ell<k$ such that for all $c\in K$ with $v(a-c)\geq \gamma$, the
value $vf_i(c)$ is fixed for each $i\geq 0$,
\begin{equation}                            \label{Kapvfi}
v(f(a)-f(c)) \>=\> vf_{\bf h}(c) + {\bf h}\cdot v(a-c)\>,
\end{equation}
and
\begin{equation}                            \label{Kapvfi2}
\dist(f(a),K)\>=\>vf_{\bf h}(c) + {\bf h}\cdot\dist(a,K)\>.
\end{equation}
Therefore, $\ndd(K(a)|K,v)\leq k$ and, modulo $vK$, all distances are multiples of $\dist(a,K)$ by powers of $p$.
\end{lemma}
\begin{proof}
Using the notions of \cite{[KuV]}, the assumption that $a$ is strongly immediate over $K$ is equivalent to
the approximation type of $a$ over $K$ being of degree $[K(a):K]$. Hence all assertions except for the
last one follow from \cite[Lemma 5.2, Proposition 7.4 and Lemma~8.2]{[KuV]} (see also \cite[Lemma~8]{[Ka]}).
For the proof of the last assertion  we use the fact that every element $b\in K(a)\setminus K$ can be written as $f(a)$ with a
nonconstant polynomial $f\in K[X]$ of degree smaller than $[K(a):K]=p^k$. Since there are exactly $k$ many distinct
${\bf h} = p^\ell$ with $0\leq\ell<k$, equation (\ref{Kapvfi2}) yields that $\ndd(K(a)|K,v)\leq k$.
\end{proof}

The following corollary shows that a uv--extension of prime
degree generated by a weakly immediate element admits exactly one distance modulo $vK$. It follows from the
previous lemma together with Lemma~\ref{EF}.

\begin{corollary}                                  \label{idegp}
Take a uv--extension $(K(a)|K,v)$ of prime degree $p$ such that $a$ is weakly
immediate over $K$. Then for every nonconstant polynomial $f\in K[X]$ of degree smaller than $p$ there is
$\gamma\in v(a-K)$ such that for all $c\in K$ with $v(a-c)\geq \gamma$, the value $vf_i(c)$ is fixed for each $i\geq 0$, and
\begin{equation}                            \label{Kapvfi1}
v(f(a)-f(c)) \>=\> vf_1(c) + v(a-c)\>.
\end{equation}
Hence for any $b\in K(a)\setminus K$,
\[
\dist(b,K)\>=\>\alpha+\dist(a,K)\quad\mbox{ for some }\alpha\in vK\;.
\]
Therefore, $\ndd^*(K(a)|K,v)=\ndd(K(a)|K,v)=1$.
\end{corollary}

\begin{proposition}                      \label{nddtow}
Assume that $(L|K,v)$ is a finite uv--extension which is a tower of
extensions of degree $p$. If $d(L|K,v)=p^m$ with $m\geq 0$, then
\[
\ndd(L|K,v)\leq m\cdot (vL:vK)\;\mbox{ and }\;\ndd^*(L|K,v)\leq m\>.
\]
\end{proposition}
\begin{proof}
We consider a tower $K=L_0\subset L_1\subset\ldots\subset L_n=L$ of
uv--extensions of degree $p$. We write $d(L_i|K,v)=p^{m_i}$, with $m_n=m$.
We proceed by induction on $i\leq n$.

The induction start is covered by Corollary~\ref{idegp} if $(L_1|K,v)$ is immediate. In this case, we have
$(vL_1:vK)=1$, $m_1=1$ and $\ndd(L_1|K,v)=1=m_1\cdot (vL_1:vK)$. Also, $\ndd^*(L_1|K,v)=1=m_1\,$.
If the extension is not immediate, then it is defectless (as it is of
prime degree). Hence $m_1=0$ and $\ndd(L_1|K,v)=0=m_1\cdot (vL_1:vK)$. Also, $\ndd^*(L_1|K,v)=0=m_1\,$.

\pars
Now we assume that for some $i<n$ we have already shown that
$\ndd(L_i|K,v)\leq m_i\cdot (vL_i:vK)$ and $\ndd(L_i|K,v)\leq m_i\,$. Take any $a\in L_{i+1}\setminus L_i$
which is weakly immediate over $K$. Since $[L_{i+1}:L_i]$ is prime,
we have that $L_{i+1}=L_i(a)$. By Lemma~\ref{distdist}, either $\dist(a,K)
=\dist(b,K)$ holds for some $b\in L_i\,$, or $\dist(a,K)=\dist(a,L_i)$.

Suppose that there is such an element $a$ for which the latter holds. Then
$a$ is weakly immediate over $L_i$ and by Lemma~\ref{EF}, the uv--extension $(L_{i+1}|L_i,v)$ is
immediate. Hence, $d(L_{i+1}|K,v)=d(L_i|K,v)\cdot p$, so $m_{i+1}=m_i +1$.
By Lemma~\ref{idegp}, $\ndd(L_{i+1}|L_i,v)=1$. This says that
modulo $vL_i$, all distances $\dist(a,K)$ arising in this way must be equal.
Consequently, there can be at most $(vL_i:vK)$ many that are distinct
modulo $vK$, and only one modulo $\widetilde{vK}$. This is in addition to the number of distinct distances arising
from elements in $L_i\,$. So we obtain that
\begin{eqnarray*}
\ndd(L_{i+1}|K,v)&\leq& (vL_i:vK)\,+\,m_i\cdot (vL_i:vK) \>=\> m_{i+1}\cdot (vL_{i+1}:vK)\>,\\
\ndd^*(L_{i+1}|K,v)&\leq& 1\,+\,m_i \>=\> m_{i+1}\>.
\end{eqnarray*}

Suppose now that there is no such element $a$. Then
\begin{eqnarray*}
\ndd(L_{i+1}|K,v)&=&\ndd(L_i|K,v)\>=\>m_i\cdot (vL_i:vK)\>\leq\> m_{i+1}\cdot (vL_{i+1}:vK)\>,\\
\ndd^*(L_{i+1}|K,v)&=&\ndd^*(L_i|K,v)\>=\>m_i\>\leq\> m_{i+1}\>.
\end{eqnarray*}
This completes our induction.
\end{proof}

\pars
We will now generalize this result to arbitrary finite, not necessarily
immediate, uv--extensions.

\begin{theorem}                              \label{nonhens}
Take a finite uv--extension $(L|K,v)$ and write $d(L|K,v)=p^m$ with $m\geq 0$. Then $\ndd(L|K,v)\leq m\cdot
[L:K]!/p^m$ and $\ndd^*(L|K,v)\leq m$. If in addition
$L|K$ is a normal extension, then $\ndd(L|K,v)\leq m\cdot (vL:vK)$.
\end{theorem}
\begin{proof}
First, we show that we may assume $(K,v)$ to be henselian. For every $a\in L\setminus K$, Lemma~\ref{dKdKh} shows
that $\dist(a,K^h)=\dist(a,K)$. By Lemma~\ref{nddKK'} we obtain that $\ndd(L|K,v)\leq\ndd(L.K^h|K^h,v)\cdot
(vK^h:vK)= \ndd(L.K^h|K^h,v)$ and $\ndd^*(L|K,v)\leq\ndd^*(L.K^h|K^h,v)$.
On the other hand, Lemma~\ref{hensdef} shows that $d(L|K,v)=
d(L.K^h|K^h,v)$, and Lemma~\ref{uvdisj} yields that $[L.K^h:K^h]=[L:K]$. Since the henselization of a valued field is an immediate extension of the field, $(vL.K^h:vK^h)=(vL^h:vK^h)=(vL:vK)$. Thus, we may replace $K$ by its henselization.

\pars
We denote the normal hull of $L$ over $K$ by $N$. Since $(K,v)$ is henselian, there is a unique extension of $v$
from $L$ to $N$ and $(N|K,v)$ is again a uv--extension. Now we take $V$ to be the ramification field of
$(N|K,v)$. From Lemma~\ref{overr} we obtain that $(V|K,v)$ is a defectless uv-extension such that
$d(L.V|V,v)=d(L|K,v)=p^m$ and that $\dist(a,V)=\dist(a,K)$ for every $a\in N\setminus K$ which is weakly
immediate over $K$. From Lemma~\ref{nddKK'} we thus obtain that $\ndd(L|K,v)\leq\ndd(L.V|V,v)\cdot (vV:vK)$ and
$\ndd^*(L|K,v)\leq\ndd^*(L.V|V,v)$.
By part b) of Lemma~\ref{overr} we know that the subextension $L.V|V$ of $N|V$ is a tower of normal extensions of
degree $p$. Hence Proposition~\ref{nddtow} shows that $\ndd(L.V|V,v)\leq m\cdot (v(L.V):vV)$ and $\ndd^*(L.V|V,v)
\leq m$. Altogether we have that $\ndd^*(L|K,v)\leq\ndd^*(L.V|V,v)\leq m$ and that
\begin{eqnarray*}
\ndd(L|K,v) &\leq& \ndd(L.V|V,v)\cdot (vV:vK)\\
&\leq& m\cdot (v(L.V):vV)\cdot (vV:vK)\>=\> m\cdot (v(L.V):vK)\>.
\end{eqnarray*}

If $L|K$ is a normal extension, then $N=L$ and $V\subseteq L$. From this we get that $(v(L.V):vK)=(vL:vK)$,
which yields the second assertion of our theorem.

\pars
In the general case, we have that $d(N|K,v)\geq d(L|K,v)=p^m$ and
\[
(v(L.V):vK)\>\leq\> (vN:vK)\>\leq\> [N:K]/d(N|K,v)\>\leq\> [N:K]/p^m\>.
\]
Since $[N:K]\leq [L:K]!\,$, this yields the first assertion of our theorem.
\end{proof}

\bn
%
%
%
\section{The number of distinct distances in all Artin-Schreier defect extensions}  \label{sectdistAS}
Throughout this section, let $(K,v)$ be a field of positive characteristic $p$. As before, we assume that $v$ is
extended to the algebraic closure $\tilde{K}$ of $K$. By Zorn's Lemma, there always exists a maximal immediate
subextension $(K'|K,v)$ of the purely inseparable $(K^{1/p}|K,v)$, where $K^{1/p}=\{c^{1/p}\mid c\in K\}$.
Throughout the present and the
final section of this paper, we will assume that $K'|K$ is finite, so that its degree is $p^m$ for some $m\geq 0$.
If $K$ has finite $p$-degree $k$, that is, $[K^{1/p}:K]=[K:K^p]=p^k$ with $k\geq 0$, then $m\leq k$.

We will now apply our previous results to consider the possible distances (modulo $vK$) of
all elements that are contained in any Artin-Schreier defect extension of $(K,v)$. In view of
Corollary~\ref{idegp}, we only have to determine the distance of one
generator of such an extension. The Artin-Schreier defect
extension $(K(\vartheta)|K,v)$ with Artin-Schreier generator $\vartheta$ is called
\bfind{dependent} if there is a purely inseparable immediate extension
$(K(\eta)|K,v)$ of degree $p$ such that
\[
v(\vartheta-\eta)\> >\>v(\vartheta-c)  \quad\mbox{for all } c\in K\;.
\]
This implies that $v(\vartheta-c)=v(\eta-c)$ for all $c\in K$ and that
\[
\dist(\vartheta,K)\>=\>\dist(\eta,K)\;.
\]
We note that by assumption, $\eta\in K^{1/p}$.

\begin{proposition}                     \label{nddpi}
Under the assumptions on $(K,v)$ outlined above, $\ndd(K^{1/p}|K,v)=\ndd(K'|K,v)\leq m$.
Moreover, if $(K,v)$ is of finite $p$-degree and $d(K^{1/p}|K,v)=p^s$, then $\ndd(K^{1/p}|K,v)\leq s$.
\end{proposition}
\begin{proof}
For every $a\in K^{1/p}$ which is weakly immediate over $K$, there must be
some $b\in K'$ with $\dist(a,K)=\dist(b,K)$. Otherwise, we would obtain
that $\dist(a,K')=\dist(a,K)$ which yields that $a$ is weakly immediate over $K'$;
since $[K'(a):K']=p$, this would show by Lemma~\ref{EF} that $(K'(a)|K',v)$ and hence also
$(K'(a)|K,v)$ are immediate extensions, contradicting the maximality of $K'$.
So we have that $\ndd(K^{1/p}|K,v)=\ndd(K'|K,v)$.

Since $(K'|K,v)$ is immediate, we have that $d(K'|K,v)=[K':K]=p^m$. Hence by Proposition~\ref{nddtow},
$\ndd(K'|K,v)\leq m$, because $(vL:vK)=1$.

For the proof of the last assertion, note that if $(K,v)$ is of finite $p$-degree, then
\[
p^m=[K':K]=d(K'|K,v)\leq d(K^{1/p}|K,v)=p^s.
\]
Thus $m\leq s$.
\end{proof}

From Proposition~\ref{nddpi} together with Corollary~\ref{idegp} we obtain the following result:
\begin{proposition}
Under the assumptions on $(K,v)$ outlined in the beginning of this section, there are elements $c_1,\ldots,c_m\in K$
such that for every dependent Artin-Schreier defect extension $(K(a)|K,v)$ there is $i\in \{1,\ldots,m\}$ such that
for every $b\in K(a)\setminus K$ there is some $\alpha\in vK$ with
\[
\dist(b,K)\>=\>\alpha+\dist(c_i^{1/p},K)\>.
\]
Hence all distinct distances modulo $vK$ of elements in dependent Artin-Schreier defect extensions of
$(K,v)$ are already among the distinct distances modulo $vK$ of elements in purely inseparable defect
extensions of degree $p$ of $(K,v)$, and their number is bounded by $m$.
\end{proposition}

In order to make a statement about \emph{all} possible Artin-Schreier defect extensions $(K(a)|K,v)$, we also have
to consider the \bfind{independent} ones, that is, the ones that are not dependent. It is shown in \cite{[Ku3]} that
if $a$ is an Artin-Schreier generator of the extension, then $\dist(a,K)$ is the lower edge of some proper convex
subgroup $H$ of $\widetilde{vK}$, that is, the lower cut set of $\dist(a,K)$ is the largest initial segment of
$\widetilde{vK}$ that does not meet $H$. We summarize:

\begin{lemma}
The distances of all elements in Artin-Schreier defect extensions
$(K(a)|K,v)$ modulo $vK$ are among the lower edges of convex
subgroups of the value group $vK$ together with the distances of the
elements in $K^{1/p}$.
\end{lemma}

The \bfind{rank} of $(K,v)$,
if finite, is the number of proper convex subgroups of the value group $vK$. Putting the previous results
together, we obtain:

\begin{theorem}                         \label{r+m}
Take a valued field $(K,v)$ of finite rank $r$, satisfying the assumptions outlined in the beginning of this section.
Then the number of distinct
distances modulo $vK$ of elements in all normal defect extensions of prime degree of $(K,v)$ as well as the
number of distinct distances modulo $vK$ of elements in Artin-Schreier defect extensions of $(K,v)$ are bounded by $r+m$. In particular, if $K$ has finite $p$-degree $k$,
then this number is bounded by $r+k$.
\end{theorem}

\pars
For a function field $K$ over a perfect base field $K_0$, the $p$-degree
$k$ is equal to the transcendence degree $\trdeg K|K_0\,$. For a valued
function field $(K|K_0,v)$ over a trivially valued base field $(K_0,v)$, the
rank is bounded by $\trdeg K|K_0\,$. This proves Theorem~\ref{MT1}.

\bn

%
%
\section{The number of distinct distances of all elements of bounded degree}		\label{sec3.3}

\textbf{ Throughout this section we shall work under the following assumptions}, unless indicated otherwise. We
take $(K,v)$ to be a valued field of positive characteristic $p$ and finite rank $r$.


For every natural number $i$ we denote by $\ndd^*_i(K,v)$ the number of distinct distances modulo $\widetilde{vK}$ of
elements $a\in \widetilde{K}\setminus K$ satisfying the following conditions:
\begin{eqnarray}                               \label{imm_elem}
\left\{ \begin{array}{l}
\textrm{$[K(a):K]$} \leq  p^i,\\
(K(a)|K,v)\textrm{ is a uv--extension,}\\
a \textrm{ is weakly immediate over }K.
\end{array} \right.
\end{eqnarray}
We will show now that for every $i\in\N$ the number $\ndd^*_i(K,v)$ is finite.

\begin{theorem}                             	\label{ndd_i}
Assume additionally that $(K,v)$ has finite $p$-degree and $d(K^{1/p}|K,v)=p^m$. Then $\ndd^*_i(K,v)$ is
finite for every natural number $i$. More precisely,
\[
\ndd^*_i(K,v) \leq r+im.\]
\end{theorem}
\begin{proof}
In what follows, let $a\in \widetilde{K}$ satisfy the assumptions (\ref{imm_elem}). Lemma~\ref{dKdKh} shows
that $\dist(a,K)=\dist(a,K^h)$. This implies in particular that $a$ is weakly
immediate over $K^h$. Furthermore, the assumptions (\ref{imm_elem}) together with Lemma~\ref{uvdisj} yield that
$[K^h(a):K^h]=[K(a):K]$. Hence, for every natural number $i$ we have that $\ndd^*_i(K,v)\leq \ndd^*_i(K^h,v)$.

We wish to show that also $(K^h,v)$ satisfies the assumptions stated at the beginning of this section.
Since $K^h|K$ is a separable algebraic extension, $K^h$ has the same $p$-degree as $K$, so $[(K^h)^{1/p}:K^h]=p^k$.
It follows that $(K^h)^{1/p}=K^{1/p}.K^h$. Since $K^{1/p}|K$ is finite and $v$ extends uniquely from $K$
to $K^{1/p}$, Lemma~\ref{hensdef} yields that
\[
p^m\>=\>d(K^{1/p}|K,v)\>=\>d(K^{1/p}.K^h|K^h,v)\>=\>d((K^h)^{1/p}|K^h,v)\>.
\]
Furthermore, $vK^h=vK$ is again of rank $r$. Hence we can assume that $(K,v)$ is henselian.

\pars
Take $K^r$ to be the absolute ramification field of $K$ with respect to the fixed extension of $v$ to $\tilde{K}$.
Lemma~\ref{overr} shows
that $\dist(a,K)=\dist(a,K^r)$. This implies in particular that $a$ is weakly
immediate over $K^r$. Moreover, $[K^r(a):K^r]\leq [K(a):K]$ and $\widetilde{vK}=\widetilde{vK^r}$. Therefore,
$\ndd^*_i(K,v)\leq \ndd^*_i(K^r,v)$ for every $i\in\N$.

We wish to show that also $(K^r,v)$ satisfies the assumptions stated at the beginning of this section.
Since $K^r|K$ is a separable algebraic extension, $K^r$ has the same $p$-degree as $K$, so $[(K^r)^{1/p}:K^r]=p^k$.
It follows that $(K^r)^{1/p}=K^{1/p}.K^r$. Lemma~\ref{overr} yields that
\[
p^m=d(K^{1/p}|K,v)= d(K^{1/p}.K^r|K^r,v)=d((K^r)^{1/p}|K^r,v).
\]
Furthermore, $vK^r/vK$ is a torsion group, hence $vK^r$ is again of rank $r$. Hence we can assume that $K^r=K$.
Note that by Lemma~\ref{overr} this means that the extension $K(a)|K$
is a tower of normal extensions of degree $p$. In particular, it is of degree $p^t$ for some $t\in\{0,\ldots,i\}$.

\pars
We proceed by induction on $i$. The case of $i=1$ is covered by Theorem~\ref{r+m}. Now assume that $i\geq 2$ and
\[
\ndd^*_{i-1}(K,v) \leq r+(i-1)m\>.
\]
To give an upper bound for $\ndd^*_{i}(K,v)$, it is enough to consider elements of degree $p^i$ over $K$ which are
weakly immediate over $K$. Indeed, the distances of elements $a$ of degree at most $p^{i-1}$ are already counted in
$\ndd^*_{i-1}(K,v)$. Hence we assume that $[K(a):K]=p^i$.

If $a$ is not strongly immediate over $K$, then by Lemma~\ref{wisi} there is an immediate extension
$(K(b)|K,v)$ with $\dist(b,K)=\dist(a,K)$ and $[K(b):K]<[K(a):K]$. By Lemma~\ref{overr} the degree $[K(b):K]$ must be a power of $p$. We
conclude that $[K(b):K]\leq p^{i-1}$, showing that $\dist(a,K)$ is already counted in
$\ndd^*_{i-1}(K,v)$. Hence we assume that $a$ is strongly immediate over $K$. By Lemma~\ref{si} this
implies that the extension $(K(a)|K,v)$ is immediate.

\pars
Assume first that  $K(a)|K$ is purely inseparable. Then from Lemma~\ref{distdist} we deduce that
$\dist(a,K)=\dist(a,K^{1/p^{i-1}})$ or $\dist(a,K)=\dist(d,K)$ for some $d\in K^{1/p^{i-1}}$. If the latter holds,
then $d$ is weakly immediate over $K$ and therefore, $\dist(b,K)$ appears already as a distance of some weakly
immediate element of degree $\leq p^{i-1}$.
So we may assume that the former holds. Then $K^{1/p^{i-1}}(a)|K^{1/p^{i-1}}$ is a purely inseparable extension of
degree $p$ and the element $a$ is weakly immediate over $K^{1/p^{i-1}}$.
Since $d(K^{1/p^{i}}|K^{1/p^{i-1}},v)=d(K^{1/p}|K,v)=p^m$, Proposition~\ref{nddpi} shows that there are at
most $m$ distinct distances of elements of $K^{1/p^{i}}$ weakly immediate over $K^{1/p^{i-1}}$, modulo
$vK^{1/p^{i-1}}=\frac{1}{p^{i-1}}vK$, hence also modulo $\widetilde{vK}$.
This renders at most $m$ additional distinct distances $\dist(a,K)$ modulo $\widetilde{vK}$.

\pars
Assume now that $K(a)|K$ is not purely inseparable. Take $E$ to be a maximal separable subextension
of $K(a)|K$; we have that $E|K$ is nontrivial. Furthermore, $E|K$ is a tower of Galois extensions of degree $p$, as
$K(a)|K$ is a tower of normal extensions of degree $p$. This shows that $K$ admits an Artin-Schreier extension
$K(\vartheta)\subseteq K(a)$, where $\vartheta $ is an Artin-Schreier generator. Since $K(a)|K$ is an immediate
extension of henselian fields, the same holds for $K(\vartheta)|K$ and thus $K(\vartheta )|K$ is an Artin-Schreier
defect extension. Take a polynomial $f\in K[X]$ such that $\vartheta=f(a)$ with $\deg f<p^i$.
Since $a$ is strongly immediate by assumption, we can apply Lemma~\ref{Kap} to obtain that
\begin{eqnarray}		\label{dist_as_def}
\dist(\vartheta ,K)\>=\> \dist(f(a),K)\>=\> \alpha + p^s\dist(a,K)
\end{eqnarray}
for some $\alpha \in vK$ and $s<i$. Take $c\in K$ such that $vc=\alpha$.

Assume that the Artin-Schreier defect extension $(K(\vartheta)|K,v)$ is dependent. Then $\dist(\vartheta,K)=
\dist(\eta,K)$ for some $\eta \in K^{1/p}$ such that the extension $K(\eta)|K$ is immediate. Hence,
\[
p^s\dist(a,K)\>=\> \dist(\vartheta ,K)-vc\>=\>\dist(\eta,K)-vc \>=\>
\dist\left(\frac{\eta}{c},K\right)\>,
\]
where the last equation holds by (\ref{shiftcut}).
Since $\frac{1}{p^s}v(\frac{\eta}{c}-K)= v\left((\frac{\eta}{c})^{1/p^s}-K^{1/p^s}\right)$, we obtain that
\begin{eqnarray}		\label{dist_dep_AS}
\dist(a,K)\>=\> \dist\left(\left(\frac{\eta}{c}\right)^{1/p^s},\,K^{1/p^s}\right)\>.
\end{eqnarray}
Since $v(a-K)$ has no maximal element, it follows from equation (\ref{dist_dep_AS}) that also
$v\left((\frac{\eta}{c})^{1/p^s}-K^{1/p^s}\right)$ has no maximal element, so $(\frac{\eta}{c})^{1/p^s}$ is weakly
immediate over $K^{1/p^s}$. Moreover, $K^{1/p^s}((\frac{\eta}{c})^{1/p^s})|K^{1/p^s}$ is a purely inseparable
extension of degree $p$. Hence, $\dist\left((\frac{\eta}{c})^{1/p^s},K^{1/p^s}\right)$ has already been counted
under the purely inseparable case in this or an earlier induction step (depending on the value of $s<i$).

\pars
Assume now that $K(\vartheta)|K$ is an independent Artin-Schreier defect extension. Then
\cite[Proposition~4.2]{[Ku6]} together with Equation~(\ref{dist_as_def}) shows that
\[
p^s\dist(\vartheta ,K)\>=\>\dist(\vartheta ,K)\>=\>vc+ p^s\dist(a,K)
\]
and consequently,
\[
\dist(a,K)=-\frac{1}{p^s}vc+ \dist(\vartheta, K).
\]
This shows that $\dist(a,K)$ is equal modulo $\widetilde{vK}$ to the distance of some weakly immediate element of
degree $p$ over $K$, which has already been counted in $\ndd^*_1(K,v)$.

Consequently, we obtain that
\[
\ndd^*_i(K,v)\>\leq\> \ndd^*_{i-1}(K,v) + m.
\]
By induction hypothesis, it follows that
\[
\ndd^*_i(K,v)\>\leq\> r+mi.
\]
\end{proof}

\parm
An interesting special case is covered by the following result. Here the assumptions on the finiteness of the
extension $(K^{1/p}|K,v)$ and its defect are not needed.
\begin{proposition}                          \label{r+1}
Assume that $(K,v)$ has finite rank $r$ and that the perfect hull of $K$ is contained in the
completion of $(K,v)$. Then
\[
\ndd^*_i(K,v) \leq r+1
\]
for every natural number $i$. Therefore, there are at most $r+1$ distances distinct modulo $\widetilde{vK}$
of elements satisfying (\ref{imm_elem}) for arbitrary $i\in\N$.
\end{proposition}
\begin{proof}
Similar to the proof of Theorem~\ref{ndd_i}, except that in all purely inseparable cases the only possible distance
is $\infty$. In particular, there are no dependent Artin-Schreier defect extensions. Indeed, if $\vartheta$ is an
Artin-Schreier generator of an Artin-Schreier defect extension, then \cite[Corollary~2.30]{[Ku3]} yields that
$v(\eta-c)<0$ for all $c\in K$. Hence there is no $\eta\in K^{1/p}$ such that $v(\eta-c)=v(\vartheta-c)$ for all
$c\in K$.
\end{proof}

We can generalize the previous proposition by dropping the condition that for each considered algebraic element $a$,
$(K(a)|K,v)$ is a uv--extension. If $H$ is a proper convex subgroup of $\widetilde{vK}$, then $H^+$ denotes the cut
at the upper edge of $H$, that is, its upper cut set is the largest final segment of $\widetilde{vK}$ which does not
meet $H$.
\begin{corollary}
Under the assumptions of Proposition~\ref{r+1}, there are at most $2r$ distances
distinct modulo $\widetilde{vK}$ of elements in $\tilde{K}$ that are weakly immediate over $K$.
\end{corollary}
\begin{proof}
Assume that $a$ is weakly immediate over $K$. Then $\dist(a ,K)=\dist(a,K^h)$ or $\dist(a,K)=\dist(d,K)$ for some
$d\in K^h$.

In the first case, we obtain that $a$ is weakly immediate over $K^h$. Hence $a$ satisfies conditions
(\ref{imm_elem}) for some $i\in\N$ with $K^h$ in place of $K$. Now if $(K,v)$ satisfies the assumptions of
Proposition~\ref{r+1}, then so does its henselization: first of all, they have the same rank, and secondly,
$(K^h)^{1/p^\infty}=K^h.K^{1/p^\infty}\subseteq K^h.K^c\subseteq (K^h)^c$. Applying Proposition~\ref{r+1}, we see
that the number of distances distinct modulo $\widetilde{vK}$ of such elements $a$ is bounded by $r+1$.

In the second case, $a$ is weakly distinguished over $K$, that is,
\[
\dist(a,K) = \alpha +H^+
\]
for some $\alpha \in vK$ and a nontrivial convex subgroup $H$ of $vK$ by \cite[Theorem~1]{[Ku4]}. Note that if
$H=vK$, we have that $\dist(a,K) = \infty$ and this distance has already been counted above.  This gives $r-1$
additional possible distances modulo $\widetilde{vK}$.

Hence we have at most $(r+1)+(r-1)=2r$ distances distinct modulo
$\widetilde{vK}$ of weakly immediate algebraic elements over $K$.
\end{proof}

\bn

\end{document}